\documentclass[12pt]{amsart}
\usepackage{amssymb}
\usepackage{amsmath, amscd}
\usepackage{amsthm}
\usepackage[table,xcdraw]{xcolor}
\usepackage{ulem}
\usepackage{tikz}
\usepackage{circuitikz}
\usepackage[colorlinks=true, linkcolor=blue,urlcolor=blue]{hyperref}

\newtheorem{theorem}{Theorem}[section]

\newtheorem{proposition}[theorem]{Proposition}
\newtheorem{lemma}[theorem]{Lemma}

\newtheorem{corollary}[theorem]{Corollary}
\theoremstyle{definition}

\newtheorem{example}[theorem]{Example}

\begin{document}

\author[P. Danchev]{Peter Danchev}
\address{Institute of Mathematics and Informatics, Bulgarian Academy of Sciences, 1113 Sofia, Bulgaria}
\email{danchev@math.bas.bg; pvdanchev@yahoo.com}
\author[A. Javan]{Arash Javan}
\address{Department of Mathematics, Tarbiat Modares University, 14115-111 Tehran Jalal AleAhmad Nasr, Iran}
\email{a.darajavan@modares.ac.ir; a.darajavan@gmail.com}
\author[A. Moussavi]{Ahmad Moussavi}
\address{Department of Mathematics, Tarbiat Modares University, 14115-111 Tehran Jalal AleAhmad Nasr, Iran}
\email{moussavi.a@modares.ac.ir; mpussavi.a@gmail.com}

\title[CSNC Rings]{Rings Whose Clean and Nil-Clean Elements Have Some Clean-Like Properties}

\maketitle




\begin{abstract}
We define two types of rings, namely the so-called {\it CSNC} and {\it NCUC} that are those rings whose clean elements are strongly nil-clean, respectively, whose nil-clean elements are uniquely clean. Our results obtained in this paper somewhat expand these obtained by Calugareanu-Zhou in Mediterr. J. Math. (2023) and by Cui-Danchev-Jin in Publ. Math. Debrecen (2024), respectively.
\end{abstract}

\section{Introduction and Motivation}

Everywhere in the text of the present article, let $R$ be an arbitrary associative (but {\it not} necessarily commutative) ring having identity element, usually stated as $1$. Standardly, for such a ring $R$, the letters $U(R)$, $\rm{Nil}(R)$ and ${\rm Id}(R)$ are designed for the set of invertible elements (also termed as the unit group of $R$), the set of nilpotent elements and the set of idempotent elements in $R$, respectively. Likewise, $J(R)$ denotes the Jacobson radical of $R$. For all other unexplained explicitly terminology and notations, we refer the interested reader to the classical source \cite{L}.

In order to illustrate our achievements here, we now need to recall the necessary fundamental material as follows:

Mimicking \cite{N}, an element $a$ from a ring $R$ is called {\it clean} if there exists $e\in {\rm Id}(R)$ such that $a-e\in U(R)$. Then, $R$ is said to be {\it clean} if every element of $R$ is clean. In addition, $a$ is called {\it strongly clean} provided $ae=ea$ and, if each element of $R$ is strongly clean, then $R$ is said to {\it strongly clean} too (see \cite{HN}). On the other side, imitating \cite{NZ}, $a\in R$ is called {\it uniquely clean} if there exists a unique $e \in {\rm Id}(R)$ such that $a-e \in U(R)$. In particular, a ring $R$ is said to be {\it uniquely clean} (or just {\it UC} for short) if every element in $R$ is uniquely clean.

Generalizing these notions, in \cite{CWZ} was defined an element $a \in R$ to be {\it uniquely strongly clean} if there exists a unique $e \in {\rm Id}(R)$ such that $a-e \in U(R)$ and $ae=ea$. In particular, a ring $R$ is said to be {\it uniquely strongly clean} (or just {\it USC} for short) if each element in $R$ is uniquely strongly clean.

In a similar vein, expanding the first paragraph of the above concepts, in \cite{CZ} a ring is said to be {\it CUC} if any clean element is uniquely clean, and a ring is said to be {\it UUC} if any unit is uniquely clean.

Acting on this idea, in \cite{CDJ} were examined those rings whose nil-clean, resp. clean, elements are uniquely nil-clean by calling them as f{\it NCUNC} and {\it CUNC} rings, respectively. On the same hand, in \cite{DHM} were studied those rings whose clean elements are uniquely strongly clean by terming them as {\it CUSC} rings.

And so, motivated by the problems posed in \cite{CDJ}, our work targets somewhat to extend the presented definitions in an attractive manner. So, we come to our key instruments.

\begin{enumerate}
\item A ring $R$ is said to be a {\it CSNC} ring if every clean element of $R$ is strongly nil-clean.
\item A ring $R$ is said to be a {\it NCUC} ring if every nil-clean element of $R$ is uniquely clean.
\end{enumerate}

In this note, our focus is on the study of CSNC rings and NCUC rings. However, we show that if $R$ is an abelian ring, that is, all idempotents are central, then CUNC rings are exactly CSNC rings as well as NCUC rings are precisely NCUSC rings, i.e., all nil-clean elements are uniquely strongly  clean.

\medskip

Likewise, we show that the next diagram holds.

\medskip

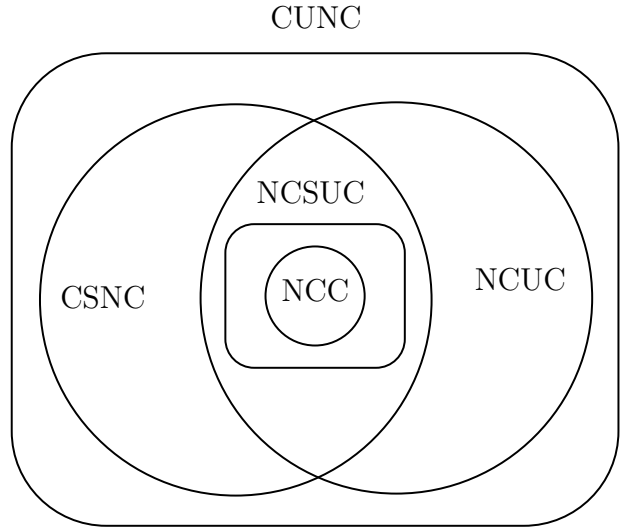
\begin{center}

\tikzset{every picture/.style={line width=0.75pt}} 

\begin{tikzpicture}[x=0.75pt,y=0.75pt,yscale=-1,xscale=1]

\draw   (157,100.7) .. controls (157,74.36) and (178.36,53) .. (204.7,53) -- (415.3,53) .. controls (441.64,53) and (463,74.36) .. (463,100.7) -- (463,243.8) .. controls (463,270.14) and (441.64,291.5) .. (415.3,291.5) -- (204.7,291.5) .. controls (178.36,291.5) and (157,270.14) .. (157,243.8) -- cycle ;
\draw   (264.75,153.75) .. controls (264.75,145.74) and (271.24,139.25) .. (279.25,139.25) -- (340.75,139.25) .. controls (348.76,139.25) and (355.25,145.74) .. (355.25,153.75) -- (355.25,197.25) .. controls (355.25,205.26) and (348.76,211.75) .. (340.75,211.75) -- (279.25,211.75) .. controls (271.24,211.75) and (264.75,205.26) .. (264.75,197.25) -- cycle ;
\draw   (252.25,176.5) .. controls (252.25,121.96) and (296.46,77.75) .. (351,77.75) .. controls (405.54,77.75) and (449.75,121.96) .. (449.75,176.5) .. controls (449.75,231.04) and (405.54,275.25) .. (351,275.25) .. controls (296.46,275.25) and (252.25,231.04) .. (252.25,176.5) -- cycle ;
\draw   (285,175.5) .. controls (285,161.69) and (296.19,150.5) .. (310,150.5) .. controls (323.81,150.5) and (335,161.69) .. (335,175.5) .. controls (335,189.31) and (323.81,200.5) .. (310,200.5) .. controls (296.19,200.5) and (285,189.31) .. (285,175.5) -- cycle ;
\draw   (171.25,177.5) .. controls (171.25,122.96) and (215.46,78.75) .. (270,78.75) .. controls (324.54,78.75) and (368.75,122.96) .. (368.75,177.5) .. controls (368.75,232.04) and (324.54,276.25) .. (270,276.25) .. controls (215.46,276.25) and (171.25,232.04) .. (171.25,177.5) -- cycle ;

\draw (291.5,165.85) node [anchor=north west][inner sep=0.75pt]   [align=left] {NCC};
\draw (179.97,169.11) node [anchor=north west][inner sep=0.75pt]  [rotate=-359.79] [align=left] {CSNC};
\draw (286,28) node [anchor=north west][inner sep=0.75pt]   [align=left] {CUNC};
\draw (389,160) node [anchor=north west][inner sep=0.75pt]   [align=left] {NCUC};
\draw (279,116) node [anchor=north west][inner sep=0.75pt]   [align=left] {NCSUC};

\end{tikzpicture}

\end{center}

\medskip

We will also demonstrate in the sequel that all NCUNC and NCUC rings are the same.

Nevertheless, to better understand the above definitions, we list the following table.

\begin{table}[h]
\begin{tabular}{
>{\columncolor[HTML]{FFFFFF}}c |
>{\columncolor[HTML]{FFFFFF}}c |
>{\columncolor[HTML]{FFFFFF}}c |
>{\columncolor[HTML]{FFFFFF}}c |
>{\columncolor[HTML]{FFFFFF}}c |
>{\columncolor[HTML]{FFFFFF}}c
>{\columncolor[HTML]{FFFFFF}}c }
{\color[HTML]{000000} }  & CUNC      & NCUC          & CSNC        & NCSUC       & NCC         \\ \hline
${\rm M}_2(\mathbb{Z}_2)$      & $\times$   & $\times$       & $\times$     & $\times$     & $\checkmark$ \\
${\rm T}_2(\mathbb{Z}_3)$      & $\times$   & $\times$       & $\times$     & $\checkmark$ & $\checkmark$ \\
${\rm T}_2(\mathbb{Z}_2)$      & $\times$   & $\times$       & $\checkmark$ & $\checkmark$ & $\checkmark$ \\
$\mathbb{Z}_3$           & $\times$   & $\checkmark$   & $\times$     & $\checkmark$ & $\checkmark$
\end{tabular}
\end{table}

\section{Main Results}

We distribute our chief results into two subsections as follows:

\subsection{CSNC rings}

We start here with the following technicality.

\begin{lemma}\cite[Corollary 2.8]{khurana2015uniquely} \label{lemma 01}
For every $f \in {\rm Id}(R)$, we have $f$ is a uniquely clean element if, and only if, $f \in {\rm Z}(R)$.
\end{lemma}

\begin{lemma}\cite[Theorem 2.1]{kocsan2016nil}\label{lemma 02}
If the element $a$ is strongly nil-clean, then $a$ is strongly clean.
\end{lemma}

\begin{proof}
If $a$ is strongly nil-clean, we write $a=e+q$ and $eq=qe$. Therefore, $a=(1-e)+((2e-1)+q)$, as required.
\end{proof}

\begin{theorem} \label{thm 2.1}
Let $R$ be a ring. Then, the following are equivalent:

\begin{enumerate}
\item $R$ is a CSNC ring.
\item For every clean element $a\in R$, the relation $a - a^2 \in {\rm Nil}(R)$ in true.
\end{enumerate}

\end{theorem}

\begin{proof}
(i) $\Rightarrow$ (ii). Assume that the element $a$ is clean. Then, there exist a nilpotent element $q\in R$ and an idempotent element $e\in R$ such that $a=e+q$ and $eq=qe$. Therefore,
$$
a-a^2=e-q-(e-q)^2=q(1-2e-q) \in {\rm Nil}(R),
$$
as claimed.

(ii) $\Rightarrow$ (i). Conversely, assume that (2) holds and that the element $a$ is clean. Thus, $a-a^2 \in {\rm Nil}(R)$. Now, let us set
$$e:=\sum_{i=0}^{n}\binom{2n}{i}a^{2n-i}(1-a)^i.$$
It can easily be now shown that $e$ is an idempotent and $a-e=h(a)(a-a^2)\in {\rm Nil}(R)$, where $h(t) \in \mathbb{Z}[t]$, as asserted.
\end{proof}

As a consequence, we extract:

\begin{corollary} \label{cor 2.2}
 Let $R$ be a ring. Then, the following are equivalent:
\begin{enumerate}
\item R is a CSNC ring.
\item For any clean element $a \in R$, there exists a unique idempotent $e \in R$ such that $a - e \in {\rm Nil}(R)$ and $ae = ea$.
\end{enumerate}
\end{corollary}

\begin{proof}
(i) $\Rightarrow$ (ii). By assumption, for any clean element $a \in R$, there exists an idempotent $e \in \mathbb{Z}[t]$ such that $a - e \in {\rm Nil}(R)$. Suppose $a - f \in {\rm Nil}(R)$ and $af = fa$, where $f \in R$ is an idempotent. Clearly, $ae = ea$, $af = fa$ and since $e \in \mathbb{Z}[t]$ and $af = fa$ we must have $ef = fe$; hence, $(a - f)(a - e)=(a - e)(a - f)$. This unambiguously shows that $e - f =(a - f) - (a - e)  \in {\rm Nil}(R)$. Apparently, $(e - f)^3 =e -f$ and so $e = f$, as desired.\\
(ii) $\Rightarrow$ (i). This implication is trivial, and so we omit the details.
\end{proof}

We now arrive at the following technicality.

\begin{proposition} \label{pro 2.3}
Let R be a CSNC ring. Then, the following conditions are valid:
\begin{enumerate}
    \item $2 \in {\rm Nil}(R)$.
    \item $J(R) \subseteq  {\rm Nil}(R)$.
\end{enumerate}
\end{proposition}

\begin{proof}
(i). Since $2$ is obviously a clean element, then by assumption $2$ can be written as a sum of an idempotent and a nilpotent, say $2=e+b$. Consequently, $1 - e = b - 1$ is both an idempotent and a unit. Therefore, $b - 1 = 1$ and thus $2 = b$ is nilpotent.\\
(ii). Suppose $a \in J(R$). So, $1+a = e + q$ for some $e\in {\rm Id}(R)$ and $q \in {\rm Nil}(R)$, where $eq=qe$. This implies that $(a +(1- e))^n= 0$ for some $n \in \mathbb{N}$. Since $a \in J(R)$, we get $1-e \in J(R)$, and so $e = 1$. Thus, $a = q \in {\rm Nil}(R)$, whence $J(R) \subseteq {\rm Nil}(R)$. Notice also that $1+a$ is a unit, so it is a clean element. Therefore, $1+a$ is also a nil-clean element.
\end{proof}

We now intend to show that a subring of a CSNC ring is again CSNC. We also prove that the finite product of CSNC rings is again CSNC. And, finally, we provide a counterexample that the infinite product of CSNC rings need {\it not} be a CSNC ring.

\begin{lemma} \label{Subring}
Let $S$ be a subring (not necessarily unital) of a ring $R$. If $R$ is CSNC, then so is $S$.
\end{lemma}

\begin{proof}
Let $a \in S$ be a clean element, writing $a = e+u$, where $u \in U(S)$ and $e \in {\rm Id}(S)$. Hence, $$(u +(1_R - 1_S))((1_R - 1_S)-u^{-1})=1_R.$$ So, $e \in {\rm Id}(R)$ and $u + (1_R - 1_S) \in U(R)$. We now put $b:= e+u+ (1_R - 1_S)$ and readily check that $b$ is clean. So, $b-b^2 \in {\rm Nil}(R)$. And since one inspects that $b-b^2=a-a^2$, we are done.
\end{proof}

As an immediate consequence, we yield:

\begin{corollary} \label{Corner ring}
Any corner ring of a CSNC ring remains a CSNC ring.
\end{corollary}

We continue with the following.

\begin{lemma}\label{finite}
A finite direct product $\prod_{i=1}^{n}R_i$ is a CSNC ring if, and only if, so is $R_i$ for each $1 \le i \le n$.
\end{lemma}

\begin{proof}
The proof is straightforward by a direct element-wise verification, so we leave the details to the interested reader.
\end{proof}

Thus, we have:

\begin{lemma}
Every finite subdirect product of CSNC rings is a CSNC ring.
\end{lemma}

The next two constructions are worthwhile.

\begin{example}
We now construct an infinite direct product of CSNC rings that is {\it not} a CSNC ring. To that goal, given $R = \prod_{k=1}^{\infty}\mathbb{Z}_{2^k}$ and choose $a = (2, 2, 2, \ldots) \in R$. Then, $a$ is a clean element, because $a=(1, 1, 1, \ldots)+(1, 1, 1, \ldots)$, but we can simply see that $a - a^2 \notin {\rm Nil}(R)$.

Arguing in a different aspect, we know by what we have proved above that in a CSNC ring the element $2$ must be nilpotent. But, an automatical check shows that, in the ring $R$, $2$ is definitely {\it not} nilpotent. So, $R$ is {\it not} a CSNC ring, as pursued.
\end{example}

\begin{example} \label{Matrix}
For any ring $R$, the matrix ring ${\rm M}_n(R)$ is {\it not} CSNC for any integer $n \ge 2$.
\end{example}

\begin{proof}
Utilizing Corollary \ref{Corner ring}, it is sufficient to demonstrate that ${\rm M}_2(R)$ is not a CSNC ring, taking into account the fact that ${\rm M}_2(R)$ is isomorphic to a corner ring of ${\rm M}_n(R)$ (for $n \ge 2$). It is, however, evident that the matrix
$A=\begin{pmatrix}
1 & 1\\
1 & 0
\end{pmatrix} \in M_2(R)$
is a clean element, but $A-A^2=I_2 \notin {\rm Nil}({\rm M}_2(R))$. Therefore, according to Proposition \ref{thm 2.1}, it can be concluded that ${\rm M}_2(R)$ cannot be a CSNC ring, as needed.
\end{proof}

We are now prepared to establish the following.

\begin{theorem}
For a ring $R$, consider the following three conditions:

\begin{enumerate}
\item $R$ is a CSNC ring.
\item $R$ is a strongly nil-clean.
\item $R/J(R)$ is boolean and $J(R)$ is nil.
\end{enumerate}

In general, $(iii)$ $\Leftrightarrow$ $(ii)$$\Rightarrow$ $(i)$. The converse holds if $R$ is semi-local.
\end{theorem}

\begin{proof}
The equivalence (ii) $\Leftrightarrow$ (iii) follows directly from \cite[Main Theorem]{DL} (see also \cite[Theorem 2.7]{kocsan2016nil}). The implication (ii) $\Rightarrow$ (i) is also elementary. Therefore, it suffices to prove (i) $\Rightarrow$ (iii).

In view of Proposition \ref{pro 2.3}, we know that $J(R)$ is nil.  As $R$ is semi-local, we my write $$R/J(R) =\bigoplus_{i=1}^{k} {\rm M}_{n_i}(D_i),$$ where $D_i$ is a division ring. Since $R/J(R)$ is a CSNC ring, for every $1 \le i \le k$, ${\rm M}_{n_i}(D_i)$ must also be a CSNC ring. Moreover, Example \ref{Matrix} insures that, for every $1 \le i \le k$, $n_i$ must be equal to $1$. Additionally, Theorem \ref{thm 2.1} establishes that any division ring that is CSNC is necessarily a Boolean ring. Consequently, we conclude that, for each $1 \le i \le k$, $D_i$ is a Boolean ring. Therefore, $R/J(R)$ is a Boolean ring, thus completing the proof.
\end{proof}

In particular, we deduce:

\begin{corollary}
For a ring $R$, consider the following five conditions:
\begin{enumerate}
\item $R$ is a CSNC ring.
\item $R$ is a strongly nil-clean ring.
\item $R$ is a uniquely nil-clean ring.
\item $R$ is a uniquely clean ring and $J(R)$ is nil.
\item $R/J(R) \cong \mathbb{Z}_2$ and $J(R)$ is nil.
\end{enumerate}

Generally, $(v)$ $\Rightarrow$ $(iv)$ $\Leftrightarrow$ $(iii)$ $\Rightarrow$ $(ii)$ $\Rightarrow$ $(i)$. The converse holds, provided $R$ is local.
\end{corollary}

We now have all the machinery necessary to prove the following statement.

\begin{theorem}
A ring $R$ is strongly nil-clean if, and only if,

\begin{enumerate}
\item $R$ is a CSNC ring.
\item $R$ is a semi-potent ring.
\end{enumerate}
\end{theorem}

\begin{proof}
($\Rightarrow$). It is clear.\\
($\Leftarrow$). Suppose $R$ is simultaneously a semi-potent ring and a CSNC ring. Then, the ring $R/J(R)$ is also semi-potent and CSNC. Since $J(R)$ is nil, employing \cite{DL} (or \cite[Theorem 2.7]{kocsan2016nil}) we get that $R$ is strongly clean if, and only if, $R/J(R)$ is strongly clean. Therefore, without loss of generality, we can assume that $J(R)=\{0\}$.

\medskip

\noindent{\bf Claim 1:} $R$ is a reduced ring, that is, ${\rm Nil}(R)=\{0\}$.

\medskip

Suppose $a^2=0$ for an arbitrary non-zero element $a\in R$. Since $R$ is semi-potent and $J(R)=\{0\}$, consulting with \cite{DL} (compare also with \cite[Theorem 2.1]{kocsan2016nil}), there exists an idempotent $e \in R$ such that $eRe={\rm M}_2(T)$ for some non-trivial ring $T$. But, in virtue of Example \ref{Corner ring}, $eRe$ is a CSNC ring. However, owing to Example \ref{Matrix}, the ring ${\rm M}_2(T)$ cannot be a CSNC ring, which is the desired contradiction. Therefore, $R$ is a reduced ring, as stated.

\medskip

\noindent{\bf Claim 2:} $R$ is boolean.

\medskip
Suppose there exists $a \in R$ such that $a^2\neq a$. Since $R$ is a semi-potent ring and $J(R)=\{0\}$, we conclude that $(a-a^2)R$ contains a non-trivial idempotent, $e$ say. Thus, we may write $e=(a-a^2)b$, where $b\in R$. Then,
$$e=e(a-a^2)b=ea\cdot e(1-a)b,$$
so $ea\in U(eRe)$. Since $eRe$ is both CSNC and reduced, we obtain $$ea-(ea)^2=ea-ea^2=e(a-a^2)\in {\rm Nil}(R)=\{0\}.$$ And since a plain check shows that $e=e(a-a^2)b$, it follows that $e=0$, which is the wanted contradiction. Finally, $R$ is a boolean ring, as formulated.
\end{proof}

As a valuable consequence, we derive:

\begin{corollary}
For any ring $R$, the following statements are equivalent:
\begin{enumerate}
\item $R$ is a strongly nil-clean ring.
\item $R$ is a nil-clean CSNC ring.
\item $R$ is a clean CSNC ring.
\item $R$ is an exchange CSNC ring.
\item $R$ is a semi-potent CSNC ring.
\end{enumerate}
\end{corollary}

Our next major assertion is the following one. For convenience, we recollect that $R$ is named a {\it UU-ring}, provided $U(R)=1+{\rm Nil}(R)$.

\begin{theorem} \label{thm 2.4}
A ring $R$ is CSNC if, and only if,
\begin{enumerate}
\item Every clean element of $R$ is strongly clean.
\item $R$ is a UU-ring.
\end{enumerate}
\end{theorem}

\begin{proof}
With the aid of Lemma \ref{lemma 02}, relation (i) holds. It is now sufficient to show that point (ii) is true. In fact, to this goal, we know that all units are clean elements. Therefore, Theorem \ref{thm 2.1} applies to get that, for any unit element $u$, the difference $u - u^2$ lies in ${\rm Nil}(R)$. Hence, $1 - u$ is in ${\rm Nil}(R)$. So, $R$ is UU, as asserted.\\
On the other hand, if we assume that $a \in R$ is clean, then $a$ is a strongly clean element. In this case, $a$ can be written as $a = (1-e)+u$, where $e \in {\rm Id}(R)$, $u \in U(R)$ and $eu=ue$. Since R is a UU-ring, $u$ must be equal to $1 + q$ for some nilpotent element $q \in {\rm Nil}(R)$ and $2\in {\rm Nil}(R)$. From this, it follows that
$$a = (1 - e) + u = e + (1 - 2e + u) = e + (2 - 2e + q),$$
where $b := 2-2e+q$ is nilpotent, because $eq=qe$, as required.
\end{proof}

Particularly, we receive:

\begin{corollary} \label{{cor 2.5}}
Let $R$ be a ring. Then, the following are equivalent:
\begin{enumerate}
\item $R$ is a CUNC ring.
\item $R$ is an abelian CSNC ring.
\end{enumerate}
\end{corollary}

\begin{proof}
In accordance with Theorem \ref{thm 2.4} and \cite[Theorem 3.1]{one}, we are set.
\end{proof}

We now proceed by proving the following technical claim.

\begin{lemma} \label{lemma 2.6}
Let $I \subseteq J(R)$ be an ideal of a ring $R$. Then, the following are equivalent:
\begin{enumerate}
\item $R$ is a CSNC ring.
\item $I$ is nil and $R/I$ is a CSNC ring.
\end{enumerate}
\end{lemma}

\begin{proof}
(i) $\Rightarrow$ (ii). With Proposition \ref{pro 2.3} at hand, $I$ is a nil-ideal. Now, we show that $\bar{R}=R/I$ is a CSNC ring. According to \cite[Theorem 2.4(1)]{khurana2015uniquely}, $\bar{R}$ is a UU ring. If $\bar{a}$ is a clean element of $\bar{R}$, then we have $\bar{a} = \bar{e} + \bar{u}$, where $\bar{e}$ is idempotent and $\bar{u}$ is invertible. So, without loss of generality, we can assume that $e$ is an idempotent in $R$ and $u$ is a unit in $R$ (note that, by assumption, $I$ is nil and $I \subseteq J(R)$). Let us define $b := e + u$. Since $R$ is a CSNC ring, we infer that $b$ is a strongly nil-clean element. However, Lemma \ref{lemma 02} is applicable to derive that $b$ is strongly clean, so $a$ is also strongly clean (notice that $\bar{a}=\bar{b}$). Now, Theorem \ref{thm 2.4} finishes the arguments.

(ii) $\Rightarrow$ (i). Assume that $a$ is a clean element of the ring $R$. Then, $\bar{a}$ is also clean, because $\bar{R}$ is a CSNC ring. Therefore, using Theorem \ref{thm 2.1}, we can write:
$$\bar{a}-\bar{a}^2 \in {\rm Nil}(\bar{R}) \Longrightarrow (a-a^2)^k \in I \subseteq {\rm Nil}(R) \Longrightarrow a-a^2 \in {\rm Nil}(R).$$
Since $a$ was an arbitrary clean element, Theorem \ref{thm 2.1} completes the argumentation.
\end{proof}

As three consequences, we obtain:

\begin{corollary} \label{cor 2.7}
Let $I$ be a nil-ideal of $R$. Then, $R$ is CSNC if, and only if, $R/I$ is CSNC.
\end{corollary}

\begin{corollary}
A ring $R$ is CSNC if, and only if, $R/{\rm Nil}_{*}(R)$ is CSNC.
\end{corollary}

\begin{corollary}
A ring $R$ is CSNC if, and only if, $2\in {\rm Nil}(R)$ and $R/2R$ is CSNC.
\end{corollary}

We now have the following.

\begin{example}
The ring $\mathbb{Z}_n$ is CSNC if, and only if, $n = 2^m$ for some $m \in \mathbb{N}$.
\end{example}

\begin{proof}
$(\Rightarrow)$. As $\mathbb{Z}_n$ is CSNC, Proposition \ref{pro 2.3} ensures that $2 \in \mathbb{Z}_n$ is nilpotent. This shows that $n = 2^m$ for some $m \in \mathbb{N}$.\\
$(\Leftarrow)$. If in Lemma \ref{lemma 2.6} we take $I=J(\mathbb{Z}_{2^m})$, then nothing remains to be proved.
\end{proof}

In particular, by what we have just shown above, the following criterion holds:

\begin{corollary} \label{cor, 2.11}
A ring $R$ is CSNC if, and only if,

\begin{enumerate}
\item $2 \in R$ is nilpotent.
\item For every clean element $a\in R$, there exists $k \ge 0$ such that $a^{2^k}- a^{2^{k+1}} \in {\rm Nil}(R)$.
\end{enumerate}
\end{corollary}

\begin{proof}
(i) $\Rightarrow$ (ii). It suffices to put $k=0$; thus by Theorem \ref{thm 2.1} the proof is completed.\\
(ii) $\Rightarrow$ (i). We show that, for every clean element $a\in R$, the difference $a-a^2 \in {\rm Nil}(R$). We have $$\left[a(1-a) \right]^{2^k}=a^{2^k}(1-a)^{2^k}=a^{2^k}(1-2c_1a+2c_2a^2- \cdots+ a^{2^k})=$$
$$=a^{2^k}(1-a^{2^k})+2g(a)=(a^{2^k}-a^{2^{k+1}})+2g(a) \in {\rm Nil}(R),$$
as expected. Note that $2 \in {\rm Nil}(R)$), where $g(t) \in \mathbb{Z}[t]$ and $c_i \in \mathbb{N}$.
\end{proof}

We are now ready to establish the following necessary and sufficient condition.

\begin{theorem}
A ring R is CSNC if, and only if,

\begin{enumerate}
\item $2 \in R$ is nilpotent.
\item For any clean element $a \in R$, there exists $k \ge 0$ such that $a^{2^k} \in R$ is strongly nil-clean.
\end{enumerate}
\end{theorem}

\begin{proof}
$(\Rightarrow)$. It is obvious.\\
$(\Leftarrow)$.  Let us assume that $a$ is an arbitrary clean element. Then, there exists $k \ge 0$ such that $a^{2^k} \in R$ is strongly nil-clean, writing $a^{2^k} = e + q$, where $e$ is an idempotent and $q$ is a nilpotent with $eq = qe$. Thus, $$a^{2^{k+1}} = e + (2e+q)q,$$ and hence $$a^{2^k}-a^{2^{k+1}}=(1-(2e+q))q \in {\rm Nil}(R).$$ The final part follows from Corollary \ref{cor, 2.11}.
\end{proof}

If $R$ is a ring and $G$ is a group, then as usual $RG$ denotes the group ring of the group $G$ over $R$. The ring homomorphism $\omega : RG \to R$, defined by $\sum r_gg \to \sum r_g$ is called the {\it augmentation map}, and $ker(\omega)$ is called the {\it augmentation ideal} of the group ring $RG$ denoted by $\Delta(RG)$. That is, $\Delta(RG) = \{\sum r_gg \in RG: \sum r_g = 0\}$.

We now need the next two technicalities.

\begin{lemma}
Let $R$ be a ring and let $G$ be a group. If $RG$ is a CSNC ring, then $R$ is also a CSNC ring and $G$ is a $2$-group.
\end{lemma}

\begin{proof}
Suppose that $RG$ is a CSNC ring. By what we have previously proved in Lemma~\ref{Subring}, it follows that $R$ is a CSNC ring as well, because $R$ is a subring of $RG$. Now, we can apply Proposition~\ref{pro 2.3} to infer that $2$ is nilpotent in both $RG$ and $R$. This routinely forces that $2RG$ is a nil-ideal of $RG$, and therefore $(R/2R)G \cong RG/(2RG)$ is a CSNC ring. Furthermore, with no harm of generality, we may assume $2 = 0 \in R$. Given any $g \in G$, we know that $g \in U(RG)$. By assumption, $1 - g$ is a nilpotent in $RG$, which means $(1- g)^{2^m} = 0$ for some $m \geq 1$. This guarantees that $$g^{2^m} - 1 = (g - 1)^{2^m} = 0,$$ and consequently $g^{2^m} = 1$. Thus, $G$ is a $2$-group, as promised.
\end{proof}

\begin{lemma}
If $R$ is a CSNC ring and $G$ is a locally finite $2$-group, then $RG$ is a CSNC ring.
\end{lemma}

\begin{proof}
The application of Proposition \ref{pro 2.3} enables us that $2 \in {\rm Nil}(R)$. Since $G$ is a locally finite $2$-group, the ideal $\Delta(RG)$ is nilpotent referring to \cite[Proposition 16(ii)]{two}. However, we observe that $R \cong RG/\Delta(RG)$ is a CSNC ring. Therefore, Corollary \ref{cor 2.7} allows us to deduce that $RG$ is a CSNC ring, as stated.
\end{proof}

Summarizing the achieved above results, we conclude:

\begin{theorem}
Let $R$ be a ring and let $G$ be a locally finite group. Then, $RG$ is a CSNC ring if, and only if, $R$ is a CSNC ring and $G$ is a $2$-group
\end{theorem}

In what follows, we try to construct some rings that are CSNC but {\it not} CUSC. To this aim, let $A, B$ be two rings and let $M, N$ be the $(A, B)$-bimodule and $(B, A)$-bimodule, respectively. Also, we consider the bilinear maps $\phi : M\otimes_B N \to A$ and $\psi : N\otimes_AM \to B$ that apply to the following properties

$$Id_M \otimes_B \psi = \phi \otimes_A Id_M, \quad Id_N \otimes_A \phi = \psi \otimes_B Id_N .$$
For $m \in M$ and $n \in N$, we define $mn := \phi(m \otimes n)$ and $nm := \psi(n \otimes m)$.
Thus, the 4-tuple
$R= \begin{pmatrix}
A & M \\
N & B
\end{pmatrix}$
forms an associative ring equipped with the obvious matrix operations, which is standardly called the {\it Morita context} ring. Denote the two-sided
ideals of $MN$ and $NM$ as $Im\phi$ and $Im\psi$, respectively, that are called the {\it trace ideals} of the Morita context.

Our next important result states thus:

\begin{theorem} \label{thm 2.20}
Let  $R= \begin{pmatrix}
A & M \\
N & B
\end{pmatrix}$
be a Morita context ring such that $MN$ and $NM$ are nilpotent ideals of $A$ and $B$, respectively. Then, $R$ is a CSNC ring if, and only if, both $A$
and $B$ are CSNC rings.
\end{theorem}

\begin{proof}
As both $A$ and $B$ are subrings of $R$, bearing in mind Lemma~\ref{Subring}, they are CSNC rings.\\

Oppositely, if $A,B$ are two CSNC rings, then Lemma \ref{lemma 2.6} informs that $J(A)$ and $J(B)$ are nil-ideals and also that $A/J(A)$ and $B/J(B)$ are both CSNC rings. On the other hand, \cite[Theorem 13(1)]{khurana2015uniquely} leads us to
$R/J(R) \cong A/J(A)\times B/J(B)$, and
$J(R)=
\begin{pmatrix}
J(A) & M \\
N & J(B)
\end{pmatrix}
$.
Since it follows from Lemma~\ref{finite} that the direct product of CSNC rings is again a CSNC ring, the ring $R/J(R)$ remains CSNC.

Furthermore, according to Lemma \ref{lemma 2.6}, it needs to show that $J(R)$ is a nil-ideal. To this purpose, suppose
$S= \begin{pmatrix}
a_1 & m \\
n & a_2
\end{pmatrix} \in J(R)$
such that $a_1^n=a_2^n=0$ for some $n \in \mathbb{N}$. Then, one inspects that
$S^n \in
\begin{pmatrix}
 MN& M \\
N & NM
\end{pmatrix}
$
and thus, after an easy manipulation by induction on $k\geq 1$, we infer that
$$\begin{pmatrix}
 MN& M \\
N & NM
\end{pmatrix}^{2k}=
\begin{pmatrix}
 (MN)^k& (MN)^kM \\
(NM)^kN & (NM)^k
\end{pmatrix}.
$$
Consequently, $J(R)$ is nil, as needed.
\end{proof}

Let $R$, $S$ be two rings, and let $M$ be an $(R, S)$-bimodule such that the operation $(rm)s = r(ms$) is valid for all $r \in R$, $m \in M$ and $s \in S$. Given such a bimodule $M$, we can set

$$
T(R, S, M) =
\begin{pmatrix}
 R& M \\
 0& S
\end{pmatrix}
=
\left\{
\begin{pmatrix}
 r& m \\
 0& s
\end{pmatrix}
: r \in R, m \in M, s \in S
\right\},
$$
where it becomes a ring with the usual matrix operations. The so-defined formal matrix $T(R, S, M$) is called a {\it formal triangular matrix} ring. In Theorem \ref{thm 2.20}, if we just set $N =\{0\}$, then we will obtain the following three consequences.

\begin{corollary}
Let $R$, $S$ be two rings and let $M$ be an $(R, S)$-bimodule. Then, the formal triangular matrix ring $T(R, S, M)$ is CSNC if, and only
if, both $R$ and $S$ are CSNC rings.
\end{corollary}

\begin{corollary}
Let $K$ be a ring, and $n \ge 1$ is a natural number. Then, ${\rm T}_n(K)$ is a CSNC ring if, and only if, $K$ is a CSNC ring.
\end{corollary}

\begin{proof}
It is sufficient just to take in the preceding assertion $R=K$, $S={\rm T}_{n-1}(K)$ and $M=K^{n-1}$.
\end{proof}

\begin{corollary}
Let $K$ be a ring, and $n \ge 1$ is a natural number. Then, ${\rm T}_n(K)$ is a CSNC ring if, and only if, $K$ is a CSNC ring.
\end{corollary}

Further, given a ring $R$ and a central element $s$ of $R$, the $4$-tuple
$\begin{pmatrix}
R & R \\
R & R
\end{pmatrix}$
becomes a ring with addition defined component-wise and with multiplication defined by
$$
\begin{pmatrix}
 a_1& x_1 \\
 y_1& b_1
\end{pmatrix}
\begin{pmatrix}
  a_2& x_2 \\
 y_2& b_2
\end{pmatrix}=
\begin{pmatrix}
 a_1a_2 + sx_1y_2& a_1x_2 + x_1b_2 \\
 y_1a_2 + b_1y_2& sy_1x_2 + b_1b_2
\end{pmatrix}.
$$
This ring is denoted by $Ks(R)$. A Morita context
$
\begin{pmatrix}
  A& M \\
 N& B
\end{pmatrix}
$
with $A = B = M = N = R$ is called a {\it generalized matrix ring} over $R$. It was observed in \cite{krylov2008isomorphism} that a ring $S$ is a generalized matrix ring over $R$ if, and only if, $S = K_s(R)$ for some $s \in {\rm C}(R)$. Here, in the present situation, $MN = NM = sR$, so that $MN \subseteq J(A) \Longleftrightarrow  s \in J(R)$, $NM \subseteq J(B) \Longleftrightarrow  s \in  J(R)$, and $MN, NM$ are nilpotents
$\Longleftrightarrow$ $s$ is a nilpotent.

Thus, Theorem \ref{thm 2.20} has the following consequence, too.

\begin{corollary}
Let $R$ be a ring, and $s \in {\rm C}(R)\cap {\rm Nil}(R)$. Then, the formal matrix ring $K_s(R)$ is CSNC if, and only if, the former ring $R$ is CSNC.
\end{corollary}

We, thereby, immediately have:

\begin{corollary}
Let $R$ be a ring. Then the formal matrix ring $K_2(R)$ is CSNC if, and only if, $R$ is CSNC.
\end{corollary}

For $n \ge 2$ and for $s \in {\rm C}(R)$, the $n \times n$ formal matrix ring over $R$, associated with $s$ and denoted by ${\rm M}_n(R, s)$, can be defined same as in \cite{tang2013class}.

\begin{corollary}
Let $R$ be a ring with $s \in {\rm C}(R) \cap {\rm Nil}(R)$ and $n \ge 2$. Then, ${\rm M}_n(R, s)$ is CSNC if, and only if, $R$ is CSNC.
\end{corollary}

\subsection{NCUC rings}

We know that idempotents are uniquely clean if, and only if, they are central. Now, a natural question arises whether idempotents are uniquely strongly clean? In the following claim we show that the answer is in the affirmative.

\begin{lemma}\label{idempo}
Each idempotent of a ring $R$ is uniquely strongly clean.
\end{lemma}

\begin{proof}
Suppose $e \in R$ is an arbitrary idempotent. Then, $e=(1-e)+(2e-1)$ is a strongly clean decomposition for $e$. Now, if $e$ has an other strongly clean decomposition of the form $e=f+u$, where $f \in {\rm Id}(R)$, $u \in U(R)$ and $fu=uf$, then $(e-f)^2 \in U(R) \cap {\rm Id}(R)$. So, $(e-f)^2=1$. Therefore, $f=(2f-1)e+1$. But since $f=(2f-1)f$, we can write:

$$f=(2f-1)f=(2f-1)((2f-1)e+1) = (2f-1)(2f-1)e+2f-1 = e+2f-1.$$
Consequently, $f=1-e$. Hence, the required decomposition is really uniquely strongly clean, as claimed.
\end{proof}

Same applies for nilpotent elements as well (compare with \cite{CWZ} too).

\begin{lemma} \label{nilpo}
Every nilpotent element of a ring $R$ is uniquely strongly clean.
\end{lemma}

\begin{proof}
Suppose $q \in {\rm Nil}(R)$. Then, $q=1+(q-1)$ is a strongly clean decomposition for $q$. Now, if $q$ has an other strongly clean decomposition of the form $q=f+u$, we conclude that $f=q-u \in U(R) \cap {\rm Id}(R) = \{1\}$. Hence, the required decomposition is really uniquely strongly clean, as asserted.
 \end{proof}

We continue our work with the following:

\begin{proposition} The following four statements are fulfilled:

\begin{enumerate}
\item Every reduced ring is an NCSUC ring.
\item Every local ring is an NCSUC ring.
\item If $R$ is a domain, then ${\rm T}_2(R)$ is a NCUSC ring.
\item For any ring $R \neq 0$, the ring ${\rm M}_2(R)$ is {\it not} a NCUSC ring.
\end{enumerate}
\end{proposition}

\begin{proof}
(i) Suppose $R$ is reduced. Then, every nil-clean element must be an idempotent, and thus by Lemma \ref{idempo} the proof is completed.\\

(ii) Suppose $R$ is local. Then, for every nil-clean element $a \in R$, there exist $q_1, q_2 \in {\rm Nil}(R)$ such that either $a=1+q_1$ or $a=0+q_2$. So, by \cite[Lemma 2.3]{kocsan2016nil} and Lemma \ref{nilpo}, $a$ must be uniquely strongly clean.

(iii) For any $a \in R$, it is clear that
$$
{\rm Id}({\rm T}_2(R)) = \left\{
\begin{pmatrix}
    0 & 0 \\
    0 & 0
\end{pmatrix},
\begin{pmatrix}
    1 & 0 \\
    0 & 1
\end{pmatrix},
\begin{pmatrix}
    1 & a \\
    0 & 0
\end{pmatrix},
\begin{pmatrix}
    0 & a \\
    0 & 1
\end{pmatrix}\right\},$$
$${\rm Nil}({\rm T}_2(R)) = \left\{ \begin{pmatrix}
    0 & a \\
    0 & 0
\end{pmatrix}\right\}.$$
Therefore, every nil-clean element of the ring ${\rm T}_2(R)$ has the following form:
$$\begin{pmatrix}
    0 & a \\
    0 & 0
\end{pmatrix},
\begin{pmatrix}
    1 & a \\
    0 & 0
\end{pmatrix},
\begin{pmatrix}
    0 & a \\
    0 & 1
\end{pmatrix},
\begin{pmatrix}
    1 & a \\
    0 & 1
\end{pmatrix}.$$
Since
$\begin{pmatrix}
    1 & a \\
    0 & 0
\end{pmatrix}$
and
$\begin{pmatrix}
    0 & a \\
    0 & 1
\end{pmatrix}$
are idempotents, they are uniquely strongly clean by Lemma \ref{idempo}. Furthermore, since
$\begin{pmatrix}
    0 & a \\
    0 & 0
\end{pmatrix}$
is nilpotent, Lemma \ref{nilpo} assures that it is uniquely strongly clean. Also, based on the existence of exactly four idempotents of the ring ${\rm T}_2(R)$, it can easily be shown that the only possible uniquely strongly clean decomposition for
$\begin{pmatrix}
    1 & a \\
    0 & 1
\end{pmatrix}$
is as follows: $$\begin{pmatrix}
    1 & a \\
    0 & 1
\end{pmatrix} = 0 + \begin{pmatrix}
    1 & a \\
    0 & 1
\end{pmatrix}.$$

(iv) Let us assume that
$A=\begin{pmatrix}
    1 & a \\
    0 & 1
\end{pmatrix} \in {\rm GL}_2(R)$. Clearly, $A$ is a nil-clean element of ${\rm M}_2(R)$. However, we also have $$A=0+A=I_2+\begin{pmatrix}
    0 & 1 \\
    1 & -1
\end{pmatrix}$$, as pursued.
\end{proof}

We now prove the following.

\begin{lemma} \label{lemma 3.1}
Every NCUC ring is abelian.
\end{lemma}

\begin{proof}
Since each idempotent element in a ring $R$ is clean, the hypothesis implies that each idempotent of $R$ is uniquely clean. Applying Lemma \ref{lemma 01}, we see that each idempotent element of $R$ is central, as required.
\end{proof}

\begin{corollary}
Every NCUC ring is NCUNC.
\end{corollary}

Our next two characterizing results are the following.

\begin{theorem}
Every CUNC ring is NCUC.
\end{theorem}

\begin{proof}
We assume that $a$ is a nil-clean element. Since $R$ is an abelian ring, it is readily verified that $a$ and $a-1$ are also clean elements. We assume that $a$ is not uniquely clean and write
$$a=e+u=f+v,$$
where $e,f$ are idempotents and $u,v$ are units. Since every CUNC ring is obviously a UU-ring, there exist $p,q \in {\rm Nil}(R)$ such that $u=1+q$ and $v = 1+p$. Therefore, $$a=e+(1+q)= f+(1+p)$$ and so $a-1=e+q= f+p$. But since $a-1$ is a clean element, it is necessarily uniquely nil-clean. However, this leads to a contradiction. Therefore, it must be that $e=f$ and hence $R$ is an NCUC ring, as stated.
\end{proof}

\begin{theorem}
Every CSNC ring is NCSUC.
\end{theorem}

\begin{proof}
We show that if $a$ is a nil-clean element, then $a$ and $a-1$ are both clean. In fact, write $a=e+q$, where $e\in {\rm Id}(R)$ and $q \in {\rm Nil}(R)$. Thus, $a-1=e+(q-1)$, and since $q-1 \in U(R)$ we obtain that $a-1$ is indeed a clean element. Also since $R$ is a CSNC ring, we have $a-1=f+p$, where $f=f^2 \in R$, $p \in {\rm Nil}(R)$ and $fp=pf$. So, $a=f+(p+1)$, which insures that $a$ is a strongly clean element. We, however, assume the contrary that $a$ is not uniquely strongly clean (notice that $a$ is a strongly clean element). So, we must have,
$$a=e+u= f+v,$$
where $e,f$ are idempotents, $u,v$ are units with $eu=ue$ and $fv=vf$. Since every ring CSNC is obviously a UU-ring, there exist $p,q \in {\rm Nil}(R)$ such that $u=1+q$ and $v= 1 + p$. Therefore, $$a=e+(1+q)= f+(1+p)$$ whence $$a-1=e+q= f+p.$$ Since $a-1$ is a clean element, one infers that both $(a-1)-e \in {\rm Nil}(R)$, $(a-1)-f \in {\rm Nil}(R)$ such that $ea=ae$ and $fa=af$. Employing Corollary \ref{cor 2.2}, we reach a contradiction. Therefore, it must be that $e=f$, as formulated.
\end{proof}

As a consequence, we yield:

\begin{corollary}
If $R$ is an NCUC ring, then every nilpotent of $R$ is uniquely nil-clean.
\end{corollary}

\begin{proof}
We show that $$({\rm Nil}(R)+{\rm Nil}(R)) \cap {\rm Id}(R)=\left\{ 0 \right\},$$ and thus the proof follows automatically from \cite[Proposition 2.1]{CDJ}. To this goal, let $e \in {\rm Nil}(R)+{\rm Nil}(R)$, so we write $e = q + p$ for some two nilpotent elements $p$ and $q$. In virtue of
Lemma \ref{lemma 3.1}, the idempotent $e$ is central, so $$q = 1 + (q - 1) = (1 - e) + [(2e - 1) - p].$$ We, thereby, have two clean decompositions for the nil-clean element $q$. Hence, $1 - e = 1$ which implies $e = 0$, as needed.
\end{proof}

We now need the following technicality.

\begin{lemma}
If $R$ is a NCUC ring, then $({\rm Nil}(R)+U(R)) \cap {\rm Id}(R)=\left\{ 1 \right\}$.
\end{lemma}

\begin{proof}
Letting $e \in {\rm Nil}(R)+U(R)$, so $e = q + u$ for $q\in {\rm Nil}(R)$ and $u \in U(R)$. So, $$q = 1 + (q - 1) = e - u.$$ We, thus, have two clean decompositions for the nil-clean element $q$. Hence, $e = 1$, as required.
\end{proof}

Our final result states the following.

\begin{theorem} \label{thm 3.5}
Let $R$ be a ring. Then, the following are equivalent:

\begin{enumerate}
\item $R$ is a NCUC ring.
\item Every strongly nil-clean element of $R$ is uniquely clean.
\item Every idempotent of $R$ is uniquely clean.
\item $R$ is abelian.
\item $R$ is NCUNC.
\end{enumerate}
\end{theorem}

\begin{proof}
(i) $\Rightarrow$  (ii). It is straightforward.

(ii) $\Rightarrow$  (iii). Knowing that each idempotent is strongly nil-clean, the implication follows immediately.

(iii) $\Rightarrow$ (iv). Follows directly from Lemma \ref{lemma 01}.

(iv) $\Leftrightarrow$ (v). Follows directly from \cite[Theorem 2.2]{CDJ}.

(iv) $\Rightarrow$ (i). Suppose $a \in R$ is nil-clean, that is, $a = e + q$, where $e \in {\rm Id}(R)$ and $q \in {\rm Nil}(R)$. Since $R$ is abelian, we can consider the strongly clean decomposition $a = (1-e) + (2e-1 + q)$. Suppose there exists $f \in {\rm Id}(R)$ and $u \in U(R)$ such that $a = f + u$. Since $R$ is abelian, we have $$(e-f)^2 = q-u \in {\rm Id}(R) \cap U(R) = \{ 1\}.$$ So, $(e-f)^2=1$ and, therefore, $f=(2f-1)e+1$.

However, since $f=(2f-1)f$, we can write:
$$f=(2f-1)f=(2f-1)((2f-1)e+1) = (2f-1)(2f-1)e+2f-1 = e+2f-1.$$
Consequently, $f=1-e$, thus completing the arguments.
\end{proof}

\medskip
\medskip
\medskip

\noindent{\bf Funding:} The first-named author, P.V. Danchev, of this research work was partially supported by the Junta de Andaluc\'ia under Grant FQM 264, and by the BIDEB 2221 of T\"UB\'ITAK.

\end{document}